\documentclass{amsart}

\IfFileExists{srcltx.sty}{\usepackage{srcltx}}

\usepackage{comment}

\numberwithin{equation}{section}

\usepackage[latin1]{inputenc}
\usepackage{xspace,amssymb,amsfonts,euscript}
\usepackage{amsthm,amsmath}
\usepackage{palatino}
\usepackage{euscript}
\input xy \xyoption {all}

\RequirePackage{color}
\definecolor{myred}{rgb}{0.75,0,0}
\definecolor{mygreen}{rgb}{0,0.5,0}
\definecolor{myblue}{rgb}{0,0,0.65}

\RequirePackage{ifpdf}
\ifpdf
  \IfFileExists{pdfsync.sty}{\RequirePackage{pdfsync}}{}
  \RequirePackage[pdftex,
   colorlinks = true,
   urlcolor = myblue, 
   citecolor = mygreen, 
   linkcolor = myred, 
   pagebackref,
   bookmarksopen=true]{hyperref}
\else
  \RequirePackage[hypertex]{hyperref}
\fi

\RequirePackage{ae, aecompl, aeguill} 

    \def\AM{{\mathbb{A}}}
    
    \def\CM{{\mathbb{C}}}

\def\GG{{\mathfrak G}}  \def\gg{{\mathfrak g}}  \def\GM{{\mathbb{G}}}
  \def\hg{{\mathfrak h}}

\def\KG{{\mathfrak K}}    
  \def\lg{{\mathfrak l}}

\def\OG{{\mathfrak O}}    
    
    \def\QM{{\mathbb{Q}}}
    
\def\SG{{\mathfrak S}}    
    \def\TM{{\mathbb{T}}}

    \def\ZM{{\mathbb{Z}}}

    \def\CC{{\mathcal{C}}}
    \def\DC{{\mathcal{D}}}
    
    \def\FC{{\mathcal{F}}}
    \def\GC{{\mathcal{G}}}

  \def\kb{{\mathbf k}}  \def\KC{{\mathcal{K}}}
    \def\LC{{\mathcal{L}}}
    
    \def\NC{{\mathcal{N}}}
    \def\OC{{\mathcal{O}}}

    \def\RC{{\mathcal{R}}}
    \def\SC{{\mathcal{S}}}

    \def\VC{{\mathcal{V}}}

\def\d{\delta}

\def\l{\lambda}
\def\L{\Lambda}

\newcommand{\nc}{\newcommand}

\def\reg{{\mathrm{reg}}}
\def\rs{{\mathrm{rs}}}

\DeclareMathOperator{\Ker}{Ker}

\def\wt{\widetilde}
\def\ov{\overline}

\def\p{{}^p}

\def\to{\rightarrow}

\nc{\ic}{\mathbf{IC}}

\nc{\gl}{{\mathfrak{gl}}}

\DeclareMathOperator{\Hom}{Hom}
\DeclareMathOperator{\supp}{supp} 
\DeclareMathOperator{\End}{End} 

\DeclareMathOperator{\id}{Id}

\newtheorem{thm}{Theorem}[section]
\newtheorem{lemma}[thm]{Lemma}

\newtheorem{prop}[thm]{Proposition}
\newtheorem{cor}[thm]{Corollary}

\theoremstyle{definition}
\newtheorem{defi}[thm]{Definition}

\theoremstyle{remark}
\newtheorem{remark}[thm]{Remark}

\DeclareMathOperator{\Ext}{Ext}

\newcommand{\into}{\hookrightarrow}

\def\uk{\underline{k}}

\def\Gr{{\EuScript Gr}}

\def\Mod{\text{-}\mathrm{Mod}}

\DeclareMathOperator{\Loc}{Loc}

\def\d{\mathbf d}
\def\gl{\mathfrak{gl}}
\def\ic{\mathbf{IC}}
\def\Grdbar{\overline{\Gr^\d}}
\def\Grcon{(\Gr^{\varpi_1})^{*d}}

\newcommand{\TMp}{\:{}^\backprime\mathbb{T}}

\DeclareMathOperator{\res}{res}

\DeclareMathOperator{\ind}{ind}

\newcommand{\indLG}{\ind_L^G}
\newcommand{\resLG}{\res_L^G}

\def\sgn{\mathrm{sgn}}

\thanks{The author was supported by an NSF postdoctoral research fellowship.}

\begin{document}

\title[A geometric Schur functor]{A geometric Schur functor}
\author{Carl Mautner}

\begin{abstract}
We give geometric descriptions of the category $C_k(n,d)$ of rational
polynomial representations of $GL_n$ over a field $k$ of degree $d$
for $d\leq n$, the Schur functor and Schur-Weyl duality.  The
descriptions and proofs use a modular version of Springer theory and
relationships between the equivariant geometry of the affine
Grassmannian and the nilpotent cone for the general linear
groups. Motivated by this description, we propose generalizations for
an arbitrary connected complex reductive group of the category
$C_k(n,d)$ and the Schur functor.
\end{abstract}

\maketitle

{\footnotesize

\par }

\section{Introduction}

\subsection{} In his 1901 thesis, Schur defined a
functor from polynomial representations of $GL_n(\CM)$ of degree $d$ to
representations of the symmetric group $\SG_d$, which is an equivalence
of categories for $n\geq d$.  Later it was observed by
Green~\cite{Green} that Schur's functor is well defined over any field
$k$, but is not an equivalence if the characteristic of $k$ is less
than or equal to $d$.  Green defines the Schur functor as follows.
Consider the group $GL_n$ and let $C_k(n,d)$ denote the category of
rational polynomial representations of $GL_n$ over $k$ of degree $d$.  The
$n$-dimensional standard representation $E$ of $GL_n$ is polynomial
and homogeneous of degree one.  Letting the symmetric group $\SG_d$
act on $E^{\otimes d}$ by permutations on the right defines a functor:
\[ \Hom(E^{\otimes d}, -): C_k(n,d) \to k\SG_d \Mod.\]

The categories $C_k(n,d)$ and  $k\SG_d \Mod$ are much more
complicated when $\mathrm{char}~k \leq d$.  The resulting categories
are not semisimple and even basic facts, for example, the dimensions
of the irreducible representations of the symmetric group, are not
known in general. The Schur functor has been used as an effective tool to relate structure
in the representation theory of the general linear groups to that of
the symmetric groups and vice versa (e.g.,~\cite{Donkin-SchurII,Donkin,Klesh}).

\subsection{} In the first part of this paper, we relate the existence
of the Schur functor to the geometry of certain singular spaces
associated to $GL_n$.   For $n\geq d$, we give a geometric
interpretation of the category $C_k(n,d)$ and the Schur functor in
terms of Springer theory for the nilpotent cone $\NC_d \subset \gl_d$.
More precisely, we show:

\begin{thm}
\label{thm-main}
For any $n \geq d$, there is an equivalence of categories:
\[ \phi^\bullet: C_k(n,d) \stackrel{\sim}{\to} P_{GL_d}(\NC_d;k),\]
where $P_{GL_d}(\NC_d;k)$ is the category of equivariant perverse
sheaves with coefficients in $k$, which takes the representation
$E^{\otimes d}$ to the Springer sheaf $\SC$.
\end{thm}

In Section~\ref{sec-SWthm}, we use this result to provide a geometric
proof of Carter-Lusztig's generalization of Schur-Weyl duality
\cite[Thm. 3.1]{CarterLusztig}.

\begin{thm}
\label{thm-SchurWeyl}
For any $d\leq n$, the morphism \[ k\SG_d \to \End_{GL_n^k}(E^{\otimes d}) \] defined by 
permuting the tensor factors is an isomorphism. 
\end{thm}

\begin{cor}
There is a commutative diagram of functors:
\begin{equation}
\xymatrix{
P_{GL_d}(\NC_d;k) \ar[rr]^{\Hom(\SC,-)} \ar[d]^\simeq_{\phi^\bullet} &&
 \End(\SC) \Mod  \\
C_k(n,d) \ar[rr]^{\Hom(E^{\otimes d},-)} && k\SG_d \Mod \ar[u]_{\simeq} ,
}
\end{equation}
\end{cor}

\subsection{} The proof of these statements is reached by crossing two bridges.

The first bridge is the geometric Satake equivalence
(cf. Section~\ref{sec-geomsat}), which relates the representation
theory of a split reductive group $G$ over $k$ to a category of
equivariant perverse sheaves on the affine Grassmannian for the
complex reductive group $G^\vee_\CM$ with dual root datum.  In
particular, it allows us to identify the category of polynomial
representations $C_k(n,d)$ with equivariant perverse sheaves on a
closed subvariety $\Grdbar$ (cf.~\ref{sec-locver}) of the affine
Grassmannian $\Gr_{GL_n}$.

The second bridge is a relationship between the nilpotent cone and the
affine Grassmannian for the group $GL_n$.  In the paper \cite{lu},
Lusztig introduced a map $\phi: \NC_d \to \Grdbar$.  Using a map in
the other direction defined on quotient stacks, we prove
(Theorem~\ref{thm-equiv}) an equivalence of categories of equivariant
perverse sheaves.

\subsection{}  In the second part of the paper, we shift our focus to
an arbitrary connected complex reductive group $G$. We observe that
our reinterpretation of the Schur functor as a functor from the
category of adjoint-equivariant perverse sheaves on the nilpotent cone
to modules over endomorphisms of the Springer sheaf is well-defined
for any reductive group.

Moreover, we give a reformulation in terms of
the Fourier-Sato transform $\TM$ on the Lie algebra $\gg$.  Let $j_\rs
: \gg_\rs \hookrightarrow \gg$ denote the open embedding of the
regular semi-simple locus and $i: \NC \hookrightarrow \gg$ the closed
embedding of the nilpotent cone.  Consider the functor 
\[\FC := j_\rs^* \circ \TM \circ i_* : P_G(\NC;k) \to P_G(\gg_\rs;k),\]
where $P_G(X;k)$ denotes the category of $G$-equivariant perverse
sheaves on a $G$-variety $X$.  We show that $\FC$ factors
through $\Loc_W(\gg_\rs)$, the category of local systems on $\gg_\rs$
with Weyl group monodromy shifted by $\dim \gg$ to be perverse, and
identify $\FC$ with the functor $\Hom(\SC,-)$.

We propose that $P_G(\NC;k)$ should be thought of as a generalization
of the category $C_k(n,d)$ and $\FC$ as the \emph{geometric Schur
  functor} for the group $G$.

\subsection{Related work}
This paper is a revised version of part of the author's Ph.D. thesis~\cite{Mautnerthesis}.  

A generalization of some of the results in this paper has since been
explored by Achar, Henderson and Riche.  In~\cite{AchHen} and then
in~\cite{AchHenRic}, they consider a split reductive group $G$ over
$k$ and the category of `small' representations of $G$ (analogous to
a variant of $C_k(n,d)$).  They show that the corresponding part of
the affine Grassmannian for the dual group $G^\vee_\CM$ contains an
open locus that maps $G^\vee$-equivariantly to the nilpotent cone for
$G^\vee$.  They use this map to construct a functor $\Psi_G$ between
categories of perverse sheaves on the affine Grassmannian and
nilpotent cone. 
The main result of these papers is an equivalence between the
composition of $\Psi_G$ with the geometric Schur functor and the
composition of the Satake equivalence with taking the zero weight
space together with its $W$-action.

Another closely-related picture arises in work of Achar and the
author.  Motivated by the equivalence between $C_k(n,d)$ and
$P_{GL_n}(\NC;k)$, we propose~\cite{AchMau} a \emph{geometric Ringel
  functor} for the equivariant derived category $D_G(\NC;k)$.  The
definition is similar to that of the geometric Schur functor in this
paper.  In particular, we define a functor 
\[ \RC := i^* \TM i_*[\dim \NC - \dim \gg]: D_G(\NC;k) \to D_G(\NC;k) \]
and show that it is an autoequivalence.

We expect that the functors $\FC$ and $\RC$ will be useful tools in
understanding the categories $P_G(\NC;k)$ and
$D_G(\NC;k)$.  Since this paper was submitted, work in this direction
for $G=GL_n$ has appeared in~\cite{AHJRgenSpr}, in which the results of
Sections~\ref{sec-GSF} and~\ref{sec-GAF} are extended to define a
stratification or iterated `recollement' of the category
$P_{GL_n}(\NC_d;k) \cong C_k(n,d)$.

\subsection{} Here is an outline of the paper.  Section \ref{sec-DP}
contains a summary of the various ingredients that will be used in the
paper.  In Section \ref{sec-proj}, we study various maps between the
nilpotent cone and affine Grassmannian and the relations they satisfy.
We then use these relations in Section \ref{sec-equiv} to prove an
equivalence of categories of perverse sheaves, with which, in Section
\ref{sec-SWthm}, we deduce Carter-Lusztig's version of Schur-Weyl
duality as a corollary.  Sections \ref{sec-GSF} and \ref{sec-GAF}
contain a proposal for a `geometric Schur functor.' 

\subsection{Acknowledgments}
This paper has been a long time coming and so the author had a number
of years to benefit from useful conversations and deep insights from
many people.  He would like to thank in particular: David Ben-Zvi for
continued support and advice, Daniel Juteau whose thesis was a source
of inspiration for much of this paper, Pramod Achar for encouragement,
Geordie Williamson for comments on a draft of the paper and an
anonymous referee for a careful reading and many helpful
comments. Thanks as well to Dennis Gaitsgory, Joel Kamnitzer, David
Helm, David Nadler, Catharina Stroppel, Zhiwei Yun, and Xinwen Zhu.

\section{Dramatis Personae}
\label{sec-DP}

\subsection{Notation} 

For a scheme $X$
defined over the complex numbers with an action of a reductive group
$G$, we denote by $P_G(X;k)$ (respectively $D_G(X;k)$) the category of
$G$-equivariant perverse sheaves (respectively the $G$-equivariant
derived category) with coefficients in $k$ on $X$. This is equivalent
to the category of perverse sheaves (respectively the constructible
derived category) on the quotient stack $[X/G]$
(cf.~\cite[Rmk. 5.5]{LO}).  For a locally closed and
$G$-invariant subscheme $Y\subset X$ and a $G$-equivariant irreducible
local system $\LC$ on $Y$ shifted to be perverse, we denote by
$\ic(Y,\LC) \in P_G(X;k)$ the simple perverse sheaf defined by the
intersection cohomology complex on $Y$ with coefficients in $\LC$.  We
denote the constant sheaf on $X$ by $\uk_X$.

\medskip

We denote the general linear group $GL_r$ by $G_r$ and its Lie algebra
by $\gg_r$.  We consider the group $G_r$ over the field $k$. We fix
the standard upper triangular Borel subgroup $B_r$ of $G_r$, with its
unipotent radical $U_r$.  Let $T_r \cong \GM_m^r$ be the Cartan
subgroup of diagonal matrices and $\hg_r \cong \AM^r$ its Lie algebra,
the standard Cartan subalgebra.  The Weyl group $W_r=\SG_r$ acts on
$\hg_r=\AM^r$ by permuting its coordinates.  We denote the weight
lattice of $G_r$ by $\Lambda_r$ and identify it with $\ZM^r$ using our
identification of the Cartan subgroup.  The set of dominant weights
$\L^+_r$ are those $\l=(\l_1,\ldots, \l_r)\in\L_r$ such that $\l_i \ge
\l_j$ for all $i<j$.

For $n$ and $d$ positive integers, we denote the category of rational homogeneous polynomial
representations of $G_n$ of degree $d$ by $C_k(n,d)$
(cf.~\cite[A.3]{Jan} for the definition). Let $\L(n,d)\subset \L_n$ be
the set of weights $(\l_1,\l_2, \ldots, \l_n)$ such that $\sum_i
\l_i=d$ and $\l_i \ge 0$ for all $1\le i \le
n$. As explained in~\cite[Prop. A.3]{Jan}, if follows from work of
Donkin~\cite{Donkin-SchurII} that the
category $C_k(n,d)$ coincides with the full subcategory of those
rational representations all of whose weights lie in the subset
$\Lambda(n,d)$.

We now fix $n\geq d$.
In what follows, we will give a geometric description of the category
$C_k(n,d)$ that does not depend on $n$.  It
is well-known that for any $n>d$, there is an equivalence $C_k(n,d)
\cong C_k(d,d)$ (see \cite[Thm 4.3.6]{Martin}).

\begin{remark}
Many of the proofs and results of this paper hold when the field $k$ is replaced by any Noetherian commutative ring $\kb$ of finite global dimension.

In this more general setting, $C_\kb(n,d)$ should be interpreted as the full
subcategory of rational representations $M$ of $GL_n$ over $\kb$ such
that the weight module $M_\l = 0$ for all $\l \notin \Lambda(n,d)$.  As
mentioned above, when $\kb$ is a field this is known to
coincide with the category of polynomial representations of
degree $d$.  

With this interpretation of $C_\kb(n,d)$, the statements and proofs of
all of the results through Section~\ref{sec-SWthm} hold with coefficients in any
Noetherian commutative ring of finite global dimension.

In Sections~\ref{sec-GSF} and~\ref{sec-GAF} we use the field
assumption explicitly.  It seems likely that this can be gotten around, for example by
using the results of~\cite{AchHenRic}.
\end{remark}

\subsection{The Schur Functor}

Using the action of the symmetric group on $E^{\otimes d}$ on the right, one
can define a functor from $C_k(n,d)$ to the category of
representations of the symmetric group $\SG_d$.

\begin{defi}
The Schur functor is defined as 
\[\Hom(E^{\otimes d},-):C_k(n,d) \to k\SG_d \Mod.\]
\end{defi}

From this definition, it is clear that it admits a left adjoint.  A
slightly different description of the Schur functor also yields a
right adjoint (cf. \cite{DEK}).   We interpret the Schur functor
geometrically and obtain similar descriptions for its adjoints as well.

\subsection{The nilpotent cone for $GL_d$} 

For any commutative $\CM$-algebra $R$, the $R$-points of $\gg_d$ is the set of endomorphisms of the free $R$-module $R^d$.  The $R$-points of the quotient stack $[\gg_d/G_d]$
is the groupoid whose objects are endomorphisms of locally free $R$-modules of rank $d$
and morphisms are isomorphisms between such pairs.

Let $\NC_d \subset \gg_d$ denote
the nilpotent cone, the variety parametrizing nilpotent endomorphisms
of $\CM^d$. For $R$ a commutative $\CM$-algebra, the set of
$R$-points of $\NC_d$ is the set of nilpotent endomorphisms of $R^d$ and the groupoid of $R$-points of $[\NC_d/G_d]$ consists of nilpotent endomorphisms of locally free $R$-modules of rank $d$ and isomorphisms between them.

We will be interested in the category $P_{G_d}(\NC_d;k)$ of equivariant perverse
sheaves on the nilpotent cone.  The $G_d$-orbits in $\NC_d$ are
equivariantly simply connected and labeled by partitions of $d$
according to the Jordan decomposition.

\subsection{The Springer and Grothendieck resolutions} 
We now list some basic facts about the Springer and Grothendieck resolutions of $\NC$ and $\gg$ for an arbitrary reductive group. For a more detailed account, see \cite[Chapter 3]{CG}.

Recall that if $\hg \subset \gg$ is a Cartan subalgebra, then Chevalley's restriction theorem tells us that $\gg/\!/G \cong \hg/W$.  Let $\chi:\gg \to \gg/\!/G$ be the quotient map.  Note that $\chi^{-1}(0)=\NC$.   Let $\hg_\reg \subset \hg$ be the complement of the reflection hyperplanes.

Recall the existence of the Grothendieck simultaneous resolution $\pi: \tilde\gg
\to \gg$, the Springer resolution $\pi_\NC:\tilde\NC \to \NC$ and the map
$\tilde\chi:\tilde\gg \to \hg$, which fit into the following diagram:
\begin{equation}
\label{diag-gg}
\xymatrix{
& \tilde\NC \ar@{^{(}->}[r]\ar[ld]_{\pi_\NC} & \tilde\gg
\ar[ld]^{\pi}\ar[rd]_{\tilde\chi} & \\
\NC \ar@{^{(}->}[r]^i\ar[rd] & \gg\ar[rd]^{\chi} & & \hg \ar[ld] \\
& \{0\} \ar@{^{(}->}[r] & \hg/W &
}
\end{equation}

The map $\pi_\NC$ is semi-small, $\pi$ is small and both are proper.  The restriction
of the map $\chi$ to $\gg_{\rs}=\chi^{-1}(\hg_\reg/W)$ is smooth with fibers isomorphic to
$G/T$.  The fundamental group of $\hg_{\reg}/W$ is the braid group $B_W$.  As
the fibers of $\chi$ over $\hg_{\reg}/W$ are simply connected, the fundamental group of
$\gg_{\rs}$ is also the braid group.  The restriction of $\pi$ to
$\tilde\gg_{\rs}$ is the $W$-cover corresponding to the subgroup of
pure braids $P_W$.  In particular, the short exact sequence $1 \to P_W
\to B_W \to W \to 1$ arises from the long exact sequence of homotopy
groups for the Galois cover $\tilde\gg_\rs \to \gg_\rs$.

\medskip

In the case $G=G_d$, the quotient $\hg/W = \AM^d/\SG_d$ is naturally isomorphic to
the affine space $Q_d$ of monic polynomials of degree $d$.  Under this isomorphism, $\chi$ is identified with the map sending an endomorphism to its characteristic polynomial.  The regular locus $\hg_\reg/W$ is identified with the open subvariety $(Q_d)_\reg \subset Q_d$ consisting of monic polynomials with $d$ distinct roots. 

\subsection{Springer theory}

Here we review Juteau's study of modular Springer theory~\cite{Ju-thesis}.  Before doing so, we remark that while Juteau works with varieties over finite fields, we consider the analogous situation over the complex numbers.

Let $\SC=\pi_{\NC*}\uk_{\tilde\NC}[\dim \tilde\NC]$ denote the Springer sheaf and
$\GC=\pi_* \uk_{\tilde\gg}[\dim \gg]$ the Groth\-en\-dieck sheaf.

There exists a \emph{Fourier-Sato transform} functor, denoted
\[\TM: D_G(\gg;k) \to D_G(\gg^*;k).\]
It is defined by composing the functor defined in~\cite[\S 3.7]{KS1}
with the shift $[\dim \gg]$.  With this shift, $\TM$ is $t$-exact for
the perverse $t$-structure~\cite[Proposition~10.3.18]{KS1}.  This
functor is an equivalence of categories with inverse
\[
\TMp: D_G(\gg^*;k) \to D_G(\gg;k).
\]

We fix an isomorphism $\gg^* \cong \gg$ and will identify them from
now on.  Brylinski~\cite[\S 11]{BRY} proves the following proposition
and Juteau~\cite{Ju-thesis} observes that it still holds in the
modular case.

\begin{prop} There is a natural isomorphism $\TM(\SC)\cong \GC$.
\end{prop}

Using this, Juteau defines a modular Springer correspondence from
irreducible representations of $W$ to $\ic$-sheaves on $\NC$. The
correspondence associates to an irreducible $W$-representation the
image of the functor $\TM j_{\rs!*}$ applied to the corresponding
local system on $\gg_\rs$.  The $\ic$-sheaves that occur in the
correspondence are precisely those contained in the top (or
equivalently in the socle) of the Springer sheaf.  To see that this is
true, note that the socle of $k[W]$ consists of the irreducible
representations of $W$ with positive multiplicities.  Thus the socle
of $j_{!*} j^* \GC=\GC$ consists of the $!*$-extensions of local
systems corresponding to irreducible $W$-representations.

\medskip

The Grothendieck sheaf $\GC$ carries a natural $W$-action because it
is a Goresky-MacPherson extension of a pushforward along a $W$-cover.
Using this action, we can equip the Springer sheaf $\SC$ with a
$W$-action in two ways: by the Fourier transform or by restriction to
the nilpotent cone.

\begin{prop}
\label{prop-signs}
  The two $W$-actions on $\SC$ --- one defined by Fourier transform
  and the other by restriction to the nilpotent cone --- differ by the
  sign character. In other words, there is an isomorphism of $W$-sheaves:
\[\TM(\GC) \cong i^* \GC \otimes \sgn.\]
\end{prop}

A version of this proposition for $\bar\QM_\ell$-sheaves appeared
in~\cite{HOT} and was proven for $\DC$-modules in~\cite{GiSpr}
and~\cite{HK}.  A proof in the modular case is given in \cite{AHJRsign}).

\subsection{The affine Grassmannian} 

We first recall the affine Grassmannian for $GL_n$, which
will play a role analogous to $\NC_d$, and then the Beilinson-Drinfeld
Grassmannian, which corresponds under the same analogy to $\gg_d$.

\subsubsection{Local version} 
\label{sec-locver}
Let $\Gr$ denote the affine
Grassmannian over $\CM$ for the group $G_n = GL_n$. In other
words, $\Gr$ is the ind-scheme over $\CM$ whose $R$-points form the set
$G_n(\KC)/G_n(\OC)$ where $\KC = R((t))$ is the ring of formal Laurent series and $\OC = R[[t]]$ is the ring of formal power series, for $R$ a
commutative $\CM$-algebra.  The group $G_n(\KC)$ acts transitively on
the set of $\OC$-lattices in $\KC^{\oplus n}$, and the stabilizer of
the standard lattice is $G_n(\OC)$.  It follows that one can view the
$R$-points of $\Gr$ as lattices.

For each $\l \in \L^+_n$, we associate the
orbit $\Gr^\lambda$ in $\Gr$ of the group scheme with $R$-points $G_n(\OC)$.
Let $\d=(d,0,\ldots,0)\in \L^+_n$.  We will be interested in $\Grdbar$
and $G_n(\OC)$-perverse sheaves supported on it.  Let $\OC_d$ denote
the quotient $\OC/t^d\OC$.  As the congruence subgroup
$\Ker(G_n(\OC)\to G_n(\OC_d))=1+t^d\gg_n(\OC)$ acts trivially on
$\Grdbar$, it is equivalent to study $P_{G_n(\OC_d)}(\Grdbar;k)$.
From the definitions above, one finds that $\Grdbar$ is a projective
variety parametrizing lattices $L$ contained in the standard lattice
$L_0=\OC^{\oplus n}$ such that the quotient $L_0/L$ is a locally free $R$-module of rank $d$.

The $G_n(\OC)$-orbits in $\Grdbar$ are labeled by the set $\L(n,d) \cap \L^+_n = \{ (\l_1,\ldots,\l_n)\in \L_n | \l_1\geq \ldots \geq \l_n \geq 0, \sum \l_i = d\}$.  As $n \ge d$, this set is in natural bijection with the set of partitions of $d$.

Let $\varpi_1$ be the fundamental weight $(1,0,\ldots,0)$. There is a $G(\OC)$-equivariant semi-small resolution $\Grcon \to
\Grdbar$.  The $\CM$-points of $\Grcon$ over a point $L \in \Grdbar$ are given by all flags
$0\subset V^1 \subset V^2 \subset \ldots \subset V^{d-1} \subset
L_0 /L$ preserved by the action of $\OC$.  Note that $\Grcon$ is smooth of dimension $(n-1)d$.

\subsubsection{Global version} 

The $d$-th Beilinson-Drinfeld Grassmannian~\cite{BD,MV} of $GL_n$, denoted by $\GG(n,d)$, is an ind-scheme defined over $Q_d$, the space of monic polynomials of degree $d$, whose
$R$-points are isomorphism classes of triples $(P,\FC,\beta)$ where
$P\in Q_d(R)$, $\FC$ is a rank $n$ vector bundle on $\AM^1$,
and $\beta$ is a trivialization of $\FC$ away from the zeros of $P$ in $\AM^1$.

Let $\KG$ be the ring of rational functions $R(t)$ and $\OG$ the
ring of polynomials $R[t]$.  Let $\LC_0$ denote the standard
$\OG$-lattice in $\KG^{\oplus n}$.

Following \cite{Ngo,MVy}, we will be interested in a particular smooth, $nd$-dimensional closed
subscheme of $\GG(n,d)$ which we denote by $\GG_d$. The $R$-points
of $\GG_d$ are the $\OG$-lattices $\LC \subset \LC_0$ such that
$\LC_0/\LC$ is a locally free $R$-module of rank $d$. To see that this
is a subfunctor of $\GG(n,d)$, note that any such lattice is a locally
free $\OG$-module of rank $n$, i.e., a vector bundle of rank $n$ on $\AM^1$. 

Equivalently, the points of $\GG_d$ can be expressed as
\[ \GG_d(R) = \{g \LC_0 \subset \LC_0 \mid g \in G_n(\KG) \cap \gg_n(\OG), \mathrm{deg}(\det(g))=d \}. \]

From this point of view, the natural map to $Q_d$ is defined by sending a lattice $\LC = g \LC_0$
to the determinant of $g$.

Let $G_{n,d}$ be the group scheme over $Q_d$ whose fiber at a
point $P\in Q_d(R)$ has $R$-points $G_{n,d}(P)(R)=GL_n(\OG/(P))$.

Ng\^{o} checks in~\cite[2.1.1]{Ngo}, that $G_{n,d} \to Q_d$ is smooth
with geometrically connected fibers of dimension $n^2d$.  It comes
with a natural action $G_{n,d}\times_{Q_d}\GG_d \to \GG_d$.

We will consider the stack $[\GG_d/G_{n,d}]$.  Its $R$-points form the
groupoid whose objects are pairs $(L \subset F)$ where $F$ is a locally free $\OG$-module of rank $n$
with $L$ a submodule, such that $F/L$ is locally free over $R$ of rank $d$.  Again,
$[\Grdbar/G_n(\OC_d)]$ is then simply the subfunctor of such pairs
where $P=x^d$.

Analogous to the Grothendieck resolution, there is a global
resolution, which we denote by $\tilde\GG_d$.  It is the scheme whose
$R$-points are full flags of $\OG$-lattices
$\OG^n = \VC^0 \supset \VC^1 \supset \ldots \supset \VC^d$ such that
$\VC^i/\VC^{i-1}$ is a locally free $R$-module of rank 1. The map
$\pi_\GG:\tilde\GG_d \to \GG_d$ is defined by 
$\pi_\GG (\VC^0\supset  \ldots \supset \VC^d) = \VC_d.$

Observe that the fibers of $\GG_d, G_{n,d}$, and $\tilde\GG_d$ over
$x^d \in Q_d$ are respectively $\Grdbar$, $G_n(\OC_d)$, and $\Grcon$.
The spaces described fit into the following diagram analogous to that of Springer theory.

\begin{equation}
\label{diag-GG}
\xymatrix{
&  \Grcon \ar@{^{(}->}[r]\ar[ld]_{\pi_\Gr} & \tilde\GG_d
\ar[ld]^{\pi_\GG}\ar[rd]_{\tilde f} & \\
\Grdbar \ar@{^{(}->}[r]\ar[rd] & \GG_d \ar[rd]^{f} & & \AM^d \ar[ld] \\
& \{x^d\} \ar@{^{(}->}[r] & Q_d &
}
\end{equation}

Similarly to above, $\pi_\Gr$ is semi-small, $\pi_\GG$ is small and both are proper.  Let $(\GG_d)_{\rs} = f^{-1}((Q_d)_{\reg})$, $(\tilde\GG_d)_{\rs} = \pi_\GG^{-1}((\GG_d)_{\rs})$, and $j_{\GG\rs}: (\GG_d)_{\rs} \to \GG_d$ be the inclusion.  The restriction $\pi_{\GG,\rs}: (\tilde\GG_d)_{\rs} \to (\GG_d)_{\rs}$ of $\pi_\GG$ is a Galois $\SG_d$-cover.  It follows that 
\[(\pi_\GG)_* \uk_{\tilde\GG_d}[nd] \cong (j_{\GG\rs})_{!*} (\pi_{\GG,\rs})_* \uk_{(\tilde\GG_d)_{\rs}} [nd].\]

\subsection{Geometric Satake}
\label{sec-geomsat}
 Mirkovi\'{c} and Vilonen~\cite{MV} prove, extending work in characteristic zero
of Lusztig~\cite{Lu2}, Beilinson-Drinfeld~\cite{BD} and Ginzburg~\cite{Gi}, that for a split
reductive group $G$ defined 
over a Noetherian commutative ring $\kb$ of finite global dimension: the category $P_{G^\vee(\OC)}(\Gr; \kb)$ of
equivariant perverse sheaves with coefficients in $\kb$ on the affine
Grassmannian for the Langlands dual group $G^\vee/\CM$, together 
with a natural convolution product $*$, is tensor equivalent to the category
of representations of $G$ over $\kb$.

In this paper, we only consider the case $G=G^\vee=GL_n$ and $\kb=k$ is a field.  We use the
following immediate corollaries of the results of~\cite{MV}.

\begin{thm}
\label{thm-geomsat}
\begin{enumerate}
\item
There is an equivalence of categories 
\[P_{G_n(\OC)}(\Grdbar;k) \cong \CC_k(n,d).\]

\item
Under this equivalence, the $GL_n$-representation
$E^{\otimes d}$ (here $E$ is the standard $n$-dimensional
representation) corresponds to the $d$-fold convolution product $(\uk_{\Gr^{\varpi_1}}[n-1])^{*d}$.  This can be expressed as the 
restriction of $\pi_{\GG*}
\uk_{\tilde{\GG}_d}[nd-d]$ to $\Grdbar$, or equivalently as the
pushforward $\pi_{\Gr*} \uk_{\Grcon}[nd-d]$.  

\item
The symmetric group $\SG_d$ acts on 
\[\pi_{\Gr*} \uk_{\Grcon}[nd-d] \cong \left( (j_{\GG,\rs})_{!*} (\pi_{\GG,\rs})_* \uk_{(\tilde\GG_d)_{\rs}} [nd-d] \right)\bigg|_{\Grdbar}  \]
by the deck transformations of $\pi_{\GG,\rs}$ and the funtoriality of $!*$-extensions.  This action corresponds under the geometric Satake equivalence to the permutation action of $\SG_d$ on $E^{\otimes d}$.
\end{enumerate}
\end{thm}

\section{A projection map and Lusztig's section}
\label{sec-proj}

In this section, we relate the spaces $\NC_d$ and
$\Grdbar$ and their quotient stacks. We do so using the functor of
points descriptions given in the previous sections.

\begin{lemma} 
\label{lem-maps}
There exist natural morphisms:
\begin{equation}
\xymatrix{
\Grdbar \ar[rd]^{\tilde\psi}\ar[d] & \NC_d
\ar[d]\ar@{_{(}->}_{\overline\phi}[l] & \GG_d
\ar[rd]^{\tilde\Psi}\ar[d] & \gg_d \ar[d]\ar@{_{(}->}_{\overline\Phi}[l]
\\ 
[\Grdbar/G_n(\OC_d)] \ar@<1ex>[r]^{\psi} & [\NC_d/G_d] \ar@{.>}[l]^{\phi}
& [\GG_d/G_{n,d}] \ar@<1ex>[r]^{\Psi} & [\gg_d/G_d]
\ar@{.>}[l]^{\Phi} 
}
\end{equation}
such that all of the solid morphisms form a commutative diagram, while the
dotted morphisms satisfy $\psi \circ \phi= \id_{[\NC_d/G_d]}$ and $\Psi \circ \Phi= \id_{[\gg_d/G_d]}$.
\end{lemma}

\begin{proof}
Recall diagrams \ref{diag-gg} and \ref{diag-GG}. In the lemma above,
the morphisms denoted by lower-case letters will be defined as the
restriction of the upper-case morphisms by the closed embeddings $\NC_d
\to \gg_d$ and $\Grdbar \to \GG_d$.  We thus need only describe the
global (or upper-case) morphisms and check that they satisfy the
relations described.

\subsection*{The maps $\tilde\Psi$ and $\Psi$} 
The map $\tilde\Psi$ is the forgetful map which associates to an
$R$-point $\LC$ of $\GG_d$ the rank $d$ locally free
$R$-module $\LC_0/\LC$ together with the endomorphism given by the action of
$t\in \OG$. Similarly, $\Psi$ associates to an $R$-point $(P,L \subset F)$ of $[\GG_d/G_{n,d}]$,
 the locally free $R$-module $F/L$ together with the 
endomorphism defined by multiplication by $t$.

\subsection*{The maps $\overline\Phi$ and $\Phi$} 
For each $R$-point $(a,E)$ in $[\gg_d/G_d]$, where $a$ is an
endomorphism of a rank $d$ locally free $R$-module $E$, let $\Phi(a)
\in [\GG_d/G_{n,d}]$ be the pair $(L \subset F)$, where $F := E[t]
\oplus \OG^{\oplus n-d}$ and $L := (a-t)E[t] \oplus \OG^{\oplus n-d}$.
To see that this is indeed a point of $[\GG_d/G_{n,d}]$, note that 
$a-t \in GL(E(t)) \cap \gl(E[t])$ and has determinant equal to the
characteristic polynomial of $a$ which has degree $d$ as desired.
The map $\overline\Phi$ is defined analogously.

\bigskip

It remains to exhibit a natural equivalence $\id \stackrel{\sim}{\to} \Psi \circ \Phi$.  

We claim that such an equivalence can be defined as follows. For any $(a,E)$, consider the map of $R$-modules given by 
\[E \into E[t] \to E[t]/(a-t)E[t].\]
Note that the map is injective as $E \cap (a-t)E[t] =0$. To prove surjectivity, one can use induction on the degree to show that any polynomial $p(t) \in E[t]$ is in $E + (a-t)E[t]$.  To finish the proof, we observe that this isomorphism intertwines the action of $a$ on $E$ with the action of $t$ on $E[t]/(a-t)E[t]$.
\end{proof}

\begin{remark}
The local version $\ov\phi$ was first observed by
Lusztig~\cite{lu}.  As far as I am aware, a description of the global
map $\ov\Phi$ first appeared in~\cite{MVy}. The maps in the other direction were known to some experts, but it seems the relationship to Lusztig's map had not been previously noticed.
\end{remark}

\begin{remark}
\label{rmk-embed}
  Lusztig~\cite{lu} observes that the map $\overline\phi$ is an open
  embedding when $n=d$ and the same is true of $\ov\Phi$.  Moreover,
  the image of the embedding intersects every orbit and provides a
  natural bijection between the orbits of $\Grdbar$ and those of $\NC_d$.
\end{remark}

\section{Equivalence of categories}
\label{sec-equiv}

In this section we use the Lemma \ref{lem-maps} from the
previous section to prove equivalences of categories, which we
then use to deduce Theorem~\ref{thm-main}.

\begin{thm}
\label{thm-equiv}
  The maps $\phi:\NC_d \to \Grdbar$ and $\Phi:\gg_d \to \GG_d$ induce equivalences of
  categories \[\phi^\bullet:=\phi^*[d^2-nd]:
  P_{G_n(\OC)}(\Grdbar) \stackrel{\sim}{\to} P_{G_d}(\NC_d),\]  \[
  \Phi^\bullet:=\Phi^*[d^2-nd] :
  P_{G_{n,d}}(\GG_d) \stackrel{\sim}{\to} P_{G_d}(\gg_d).\]
\end{thm}

\begin{proof}
By~\cite[4.2.5]{BBD}, the (shifted) pull-back along a smooth morphism $f:X \to Y$
with connected fibers is a fully faithful embedding of the
category of perverse sheaves on $Y$ to that on $X$.  

We now check that $\Phi$ enjoys these properties (and therefore
$\phi$ does as well by base change). From this it will follow that
$\Phi^\bullet$ and $\phi^\bullet$ are fully faithful.  

First consider the case $n=d$. Recall that $\overline\phi$ and $\overline\Phi$ are open
morphisms (Remark \ref{rmk-embed}).  From this, we can deduce that $\phi$ and $\Phi$ are smooth. 

By Lemma~\ref{lem-maps}, $\Psi \circ \Phi = \id_{[\gg_d/G_d]}$ and
thus the fibers of $\Phi$ are connected as long as they are nonempty.
To see that the fibers are non-empty, note that every $G_{n,d}$-orbit
of $\GG_d$ has non-trivial intersection with $\gg_d$.

In the general case $n>d$, we can factor $\tilde\Phi_d^n$ as
$\tilde\Phi_d^d$ and the map $a_{d,n}:[\GG_d/G_{d,d}] \to
[\GG_n/G_{n,d}]$. In \cite[Lemme 2.2.1]{Ngo}, it is shown that
$a_{d,n}$ is smooth with connected fibers.

By Lemma~\ref{lem-maps}, the compositions $\phi^* \circ \psi^*$ and
$\Phi^* \circ \Psi^*$ are the identity functors on $P_{G_d}(\NC_d)$ and
$P_{G_{n,d}}(\GG_d)$ respectively. We conclude that
$\phi^\bullet$ and $\Phi^\bullet$ are also essentially surjective,
which completes the proof.
\end{proof}

We can now prove Theorem~\ref{thm-main}:

\begin{proof}
Composing the equivalence from Theorem~\ref{thm-geomsat} and the first
equivalence from Theorem~\ref{thm-equiv} give the desired result.
\end{proof}

We conclude this section with an observation that we will not use in what
follows, but which puts in perspective the relationship between the various
equivariant derived categories.

\begin{cor}
  The pullback functor on equivariant derived categories
\[\psi^*: D_{G_d}(\NC_d) \to D_{G_n(\OC)}(\Grdbar)\] 
(resp. $\Psi^*: D_{G_d}(\gg_d) \to D_{G_{n,d}}(\GG_d)$)
  splits the functor 
\[\phi^*:D_{G_n(\OC)}(\Grdbar) \to D_{G_d}(\NC_d)\]
 (resp. $\Phi^*:
  D_{G_{n,d}}(\GG_d) \to D_{G_d}(\gg_d)$), in the following
  sense:

For any two objects, $A,B \in D_{G_d}(\NC_d)$, $\psi^*$ induces an injection
of graded vector spaces 
\[\Ext^*_{D_{G_d}(\NC_{d})}(A,B) \into \Ext^*_{D_{G_n(\OC)}(\Grdbar)}(\psi^*A,\psi^*B),\]
which naturally splits the projection map given by $\phi^*$.
\end{cor}

\begin{remark} 
It is worth noting that as functors between equivariant derived
categories, the pullback functors do not induce equivalences, and in
fact the various categories are not equivalent.  To see this, consider
the $\Ext$-groups between the $!$-- and $*$--extension of the constant
sheaf on the stratum associated to a partition $\lambda$.  The $\Ext$
groups in the equivariant derived categories will agree with the
equivariant cohomology of the stratum, which in turn agrees with the
group cohomology (of the reductive part) of the stabilizer of a point.
Unless $n=d$ and $\lambda$ is the trivial partition, these will not
agree.
\end{remark}

\section{Generalized Schur-Weyl Theorem}
\label{sec-SWthm}

In this section we prove Theorem~\ref{thm-SchurWeyl}.

\begin{proof}
Recall the following consequences of the geometric Satake equivalence summarized in Theorem \ref{thm-geomsat}.
There is natural isomorphism
\[ \End_{G_n}(E^{\otimes d}) \cong \End_{P_{G_n(\OC)}(\Gr)}(\pi_{\Gr *}\uk_{\Grcon}[nd-d]).\]
The perverse sheaf $\pi_{\Gr *}\uk_{\Grcon}[nd-d] \cong \left(
  (j_{\GG\rs})_{!*} (\pi_{\GG\rs})_* \uk_{\tilde\GG_{\rs}} \right)|_\Gr[nd-d]$ carries an action of $\SG_d$ induced by the action of the deck-transformations of $\pi_{\GG\rs}$, and this action agrees under the isomorphism above with the permutation action of $\SG_d$ on $E^{\otimes d}$.

We wish to translate this action from the affine Grassmannian to the nilpotent cone. A simple
generalization of the maps $\ov\Phi$ and $\ov\phi$ completes the following commutative cube:
\begin{equation}
\xymatrix{
& \tilde\NC_d \ar[r]\ar[dl]\ar[dr] & \tilde\gg_d \ar[dl]\ar[dr] & \\ 
\NC_d \ar[dr]\ar[r]^i & \gg_d \ar[dr] & \Grcon \ar[dl]\ar[r] & \wt\GG_d \ar[dl] \\
& \Grdbar \ar[r] & \GG_d &   \\
}
\end{equation}
Here the square involving $\NC$'s and $\Gr$'s and the square involving
$\gg$'s and $\GG$'s are both pull-back squares, all the
maps from upper-right to lower-left are proper, and the maps from upper-left to lower-right are inclusions.

By Theorem \ref{thm-equiv}, the functor $\phi^\bullet$ induces an isomorphism
\[ \End_{P_{G_n(\OC)}(\Gr)}(\pi_{\Gr *}\uk_{\Grcon}[nd-d]) \cong \End_{P_{G_d}(\NC_d)}(\SC).\]
The $\SG_d$-action on $\pi_{\Gr *}\uk_{\Grcon}[nd-d]$ corresponds to
the analogously-defined action on $\SC \cong i^*(j_\rs)_{!*}
\pi_{\rs*} \uk_{\tilde\gg_\rs}[d^2-d]$.

On the other hand, Juteau \cite{Ju-thesis} observes that, as in characteristic 0, the
Fourier transform exchanges the Springer and Grothendieck sheaves and
so 
\[ \End(\SC) \stackrel{\TM}{\cong} \End(\GC) \cong \End(\pi_{\rs*} \uk_{\tilde\gg_\rs}[d^2]) \cong k\SG_d.\]
 Proposition \ref{prop-signs} says that the resulting action of $\SG_d$ by
Fourier transform differs from the one arising by restriction (as in
the previous paragraph) by a sign character.  But as the first induces
an isomorphism, the second does as well.
\end{proof}

\section{Geometric Schur Functor}
\label{sec-GSF}

The previous section allows us to identify the functor $\Hom(\SC,-)$ with the Schur functor.  In other words, we have a
commutative diagram of functors:
\begin{equation}
\xymatrix{
P_{G_d}(\NC_d) \ar[rr]^{\Hom(\SC,-)} \ar[d]^\cong &&  \End(\SC) \Mod  \\
C_k(n,d) \ar[rr]^{\Hom(E^{\otimes d},-)} && k\SG_d \Mod \ar[u]_\cong ,
}
\end{equation}
where the action of $\SG_d$ on $\SC$ is through the action by deck transformations on $\GC$ and the identification of $\SC$ with $i^* \GC[-d]$.

Motivated by this connection, we will shift our focus to a general
connected complex reductive group $G$. By analogy, we can define a
\emph{Schur functor} from the category of $G$-equivariant perverse
sheaves on the nilpotent cone in $\gg$ to representations of the Weyl
group,
\[ \Hom(\SC,-): P_G(\NC_G;k) \to \End(\SC) \Mod.\]

In this section, we will introduce another functor that is closely
related to $\Hom(\SC,-)$.  We propose that this new functor should be
viewed as a geometric Schur functor for $G$.

To ease notation, we will let $j=j_{\rs}$.
We consider the functor:
\[\FC := j^* \circ \TM \circ i_* : P_{G}(\NC) \to P_{G}(\gg_{\rs}).\]

Note that it is exact because each component is exact.  Let
$\Loc_W(\gg_\rs) \subset P_G(\gg_\rs)$ denote the full subcategory of
local systems with monodromy factoring through the Weyl group shifted
to be perverse.

In the proof of the following lemma we will use parabolic restriction
and induction functors.  For their definition and some basic
properties in this context, see for example \cite{AchMau}.

\begin{lemma}
\label{lem:locW}
The functor $\FC$ factors through the inclusion $\Loc_W(\gg_\rs) \subset P_G(\gg_\rs)$.
\end{lemma}

\begin{proof}
We first show that the statement is true for all simple objects in
$P_G(\NC;k)$.  For each nilpotent orbit $\OC$ and irreducible local
system $\LC \in P_G(\OC;k)$, consider the $\ic$-sheaf $A=
\ic(\OC,\LC)$.  Let $T$ be a maximal torus for $G$.  Note that
$\res_T^G A \neq 0$ if and only if $A$ is in Juteau's modular Springer
correspondence and then $\TM A \cong \ic(\gg_\rs, \LC_V)$, where
$\LC_V \in \Loc_W(\gg_\rs)$ and corresponds to some irreducible
$k[W]$-representation $V$.

Suppose instead that $\res_T^G A = 0$.  Then there exists some Levi
subgroup $L$ such that $A_0 = \resLG A$ is non-zero but parabolic
restriction of $A_0$ for any proper Levi subgroup of $L$ vanishes.  It
follows that the same is true of $\TM A_0$ and
by~\cite[Lemma~4.4]{mirkovic}\footnote{While Mirkovi\'{c} works with
  $\DC$-modules, the same argument works in the constructible context
  with arbitrary coefficients.}, that $\TM A_0$ is supported on $\NC_L
\times Z(\lg)$.  By adjunction, there exists a non-zero map $A \to
\indLG A_0$.  As $A$ is simple, 
\[\overline{\supp \TM A} \subset \overline{\indLG \TM A_0}\subset {}^G(\NC_L \times Z(\lg)).\]
Finally, ${}^G(\NC_L \times Z(\lg)) \cap \gg_{\rs} = \emptyset$ as $L$
is not a maximal torus.  Thus $j^* \TM A =0$.

As $\FC$ is exact, we conclude that $\FC(X)$
is a local system on $\gg_\rs$ for any $X\in P_G(\NC)$.  It remains to
check that $\FC X$ has monodromy that factors through $W$.

Because we have assumed that $k$ is a field, the category of
perverse sheaves is Artinian.  Thus there exists a minimal $Q \subset
X$ such that $\FC(X/Q)=0$.  By the exactness of $\FC$, $\FC(Q)
\cong \FC(X)$.  By the minimality of $Q$, the top of $Q$ is a direct
sum of $\ic$-sheaves in Juteau's modular Springer correspondence.
Recall that the Springer sheaf is a direct sum of the projective
covers of such $\ic$-sheaves.  Thus for some $m \geq 0$, there exists
a surjection $\SC^{\oplus m} \to Q$.  Using again the exactness of
$\FC$, we conclude that $\FC (Q)$ is a quotient of $m$ copies of
$j^*\GC$ and thus has monodromy that factors through the Weyl group.
\end{proof}

In order to compare $\FC$ with $\Hom(\SC,-)$, we consider the functor 
\[ \rho: \End(j^*\GC) \Mod \to \Loc_{W}(\gg_{\rs}) \]
defined by the tensor product $(-) \otimes_{\End(j^* \GC)} j^*\GC$ and its inverse 
 \[\Hom(j^*\GC, -): \Loc_W (\gg_\rs) \to \End(j^*\GC) \Mod.\]

\begin{thm}
\label{thm-schurgeom}
There is a natural equivalence of functors:
\begin{equation}
\xymatrix{
P_{G}(\NC) \ar[rr]^{\FC} \ar[rd]_{\Hom(\SC,-)} &&
  \Loc_{W}(\gg_{\rs})\\
& \End(\SC) \Mod \cong \End(j^*\GC) \Mod \ar[ru]_{\rho}^\sim &
}
\end{equation}
\end{thm}

\begin{proof}
We proceed by constructing a sequence of equivalences. Fourier transform induces a equivalence of functors:
\[\rho(\Hom(\SC,-)) \cong \rho(\Hom(\GC,\TM i_*(-))).\]

For any $X \in P_G(\NC)$, the Fourier transform $\TM X$ is perverse and so we obtain an exact sequence:
\[0\to K \to \TM X \to \p j_*j^* \TM X \to C \to 0\]
where $K$ and $C$ have support in $\gg - \gg_\rs$.  Applying the exact functor $\Hom(\GC,-)$ gives an isomorphism $\Hom(\GC,\TM X) \cong \Hom(\GC, \p j_*j^* \TM X)$.  Using the adjunction between $j^*$ and $\p j_*$ we obtain an equivalence:
\[ \rho(\Hom(\GC,\TM i_*(-))) \cong \rho(\Hom(j^* \GC, j^* \TM i_*(-))). \]

The composition $\rho \circ \Hom(j^*\GC,-)$ restricted to $\Loc_W(\gg_\rs)$ is isomorphic to the identity functor. From Lemma \ref{lem:locW}, the functor $j^* \TM i_*$ takes values in $\Loc_W(\gg_\rs)$ and so we may conclude:
\[\rho(\Hom(j^* \GC, j^* \TM i_*(-))) \cong j^* \TM i_*(-) = \FC(-).\]
\end{proof}

\section{Geometric Adjoint Functors}
\label{sec-GAF}

In the previous section we introduced the notion of a geometric Schur functor for an arbitrary connected complex reductive group.  We close with a simple observation in this general context.

By the usual adjointness for closed and open embeddings, the geometric Schur functor has the following left and right adjoints which we consider restricted to $\Loc_W(\gg_\rs)$:
\[\GC_R = \p i^! \circ \TMp \circ \p j_* :\Loc_W(\gg_\rs) \to P_G(\NC),\]
\[\GC_L = \p i^* \circ \TMp \circ \p j_! : \Loc_W(\gg_\rs) \to P_G(\NC).\]

\begin{lemma}
\label{lem-antiorb}
The functor $\TMp j_{!*}: \Loc_W(\gg_{\rs}) \to P_{G}(\gg)$ factors through the subcategory $P_{G}(\NC)
\subset P_{G}(\gg)$.
\end{lemma}

In other words, the perverse extension of a local system associated to
a representation of the Weyl group is \emph{anti-orbital}.

\begin{proof}
Since $P_G(\NC)$ is a thick subcategory of $P_G(\gg)$ and $j_{!*}$
preserves epimorphisms, it suffices to check on a projective generator
of $\Loc_W(\gg_\rs)$.  The local system $j^*\GC$ is such a projective
generator and $j_{!*}j^*\GC \cong \GC$.  Thus $\TMp j_{!*} j^*\GC
\cong i_* \SC$, which has support on $\NC$.
\end{proof}

\begin{prop}
The adjoint functors $\GC_L,\GC_R$ are both right inverses to $\FC$.
\end{prop}

\begin{proof}
We provide the proof for $\GC_R$, the proof for $\GC_L$ is
completely analogous.

We seek a natural equivalence $j^* \TM i_* \p i^! \TM \p j_*
\cong \id$.   Consider an element $\LC \in \Loc_W(\gg_\rs)$. One has the canonical inclusion $j_{!*}\LC \hookrightarrow \p j_* \LC$.
The Fourier transform $\TMp$ is an equivalence of categories, so
$\TMp j_{!*}\LC \hookrightarrow \TMp \p j_* \LC$.  By lemma
\ref{lem-antiorb}, the $\TMp j_{!*}\LC$ is supported on the nilpotent
cone $\NC$.
On the other hand, for any $\FC\in P(\gg)$, $i_* \p
i^! \FC$ is the largest subobject of $\FC$ with support in $\NC$.  

In particular, letting $\FC = \TMp (\p j_* \LC)$ we obtain a commutative triangle:

\begin{equation}
\xymatrix{
\TMp j_{!*} \LC \ar@{^{(}->}[rr]\ar@{_{(}->}[rd] && \TMp \p j_* \LC \\
& i_* \p i^! \TMp \p j_* \LC \ar@{^{(}->}[ur]&
}
\end{equation}
We now apply the exact functor $j^* \circ \TM$ to obtain:
\begin{equation}
\xymatrix{
\LC \ar[rr]^\sim \ar@{_{(}->}[rd] && \LC \\
& j^* \TM i_* \p i^! \TMp \p j_* \LC \ar@{^{(}->}[ur]&
}
\end{equation}
We conclude that the two bottom maps are also isomorphisms and thus
$j^* \circ \TM$ applied to the adjunction morphism gives the natural
equivalence we sought to prove.
\end{proof}

\bibliographystyle{myalpha}
\bibliography{gen}

\end{document}